\documentclass[preprint]{siamltex}

\usepackage[leqno]{amsmath}
\usepackage{amssymb}
\usepackage{graphicx}
\usepackage{color}
\usepackage{multirow}
\newtheorem{remark}{Remark}
\renewcommand{\d}{\,\mathrm{d}}
\DeclareMathOperator{\sn}{sn}
\DeclareMathOperator{\sgn}{sgn}
\DeclareMathOperator{\ind}{ind}

\DeclareMathOperator{\K}{K}
\DeclareMathOperator{\spann}{span}
\DeclareMathAlphabet{\mathbf}{OT1}{cmr}{bx}{it}

\newcommand{\TSec}[2]{#1_{#2}}
\newcommand{\BSec}[2]{#1_{{#2}'}}
\newcommand{\sTSec}[2]{#1^*_{#2}}
\newcommand{\sBSec}[2]{#1^*_{{#2}'}}

\begin{document}

\title{Zolotarev quadrature rules and load balancing for the FEAST eigensolver}

\author{Stefan G\"{u}ttel\thanks{School of Mathematics, The University of Manchester, Alan Turing Building, Oxford Road, M13\, 9PL Manchester, 
United Kingdom, \texttt{stefan.guettel@manchester.ac.uk}}
\and
Eric Polizzi\thanks{Department of Electrical and Computer Engineering, 
University of Massachusetts, Amherst, USA, \texttt{polizzi@ecs.umass.edu}. This  author was partially 
supported by the National Science Foundation under Grant \#ECCS-0846457,
and also acknowledges travel support from the EPSRC Network Grant EP/I03112X/1.}
\and
Ping Tak Peter Tang\thanks{Intel Corporation, USA, \texttt{peter.tang@intel.com}}
\and
Gautier Viaud\thanks{Ecole Centrale Paris, France,  \texttt{gautier.viaud@ecp.fr}}}

\date{}

\maketitle

%
% questions of sg:
% - in principle it should be enough to have the search space just one larger than the number of evs?
% - how do we justify that in the num. ex. we always use S = inf for Gauss and Trapez?
%   using FEAST v. 2.1 only supports circle, rely on the fact that value is +/- 1/2
%   adding another parameter
%   conclusions: all the new options expected to be included in v.3
% - should we show one Zolotarev example where m = 32 or so, so we only need 1 iteration?
% - I've added acknowledgement to the na-hpc network grant, okay?
% - references to beyn, sakurai, kressner, trefethen, Polizzi/Tang, ...
% - add sentence about computation of zolo weights/poles/absterm
% - add +1 to the iteration numbers

%%%%%%%%%%%%%%%%%%%%%%%%%%%%%%%%%%%%%%%%%%%%%%%%%%%%%%%%%%%%%%%%%%%%

\begin{abstract}
The FEAST method for solving large sparse eigenproblems is equivalent to subspace iteration with an approximate spectral projector and implicit orthogonalization.
This relation allows to characterize the convergence of this method in terms of the error of a certain rational approximant to an indicator function.
We propose improved rational approximants leading to
FEAST variants with faster convergence, in particular,
when using rational approximants based on the work of Zolotarev.
Numerical experiments demonstrate the possible computational savings especially for pencils whose eigenvalues are not
 well separated and when the dimension of the search space is only slightly larger than the number of wanted eigenvalues.
%This allows for the minimization of communication in the FEAST method as well as
%for load balancing.
%This allows for increase in robustness and capabilities to achieve load balancing in parallel applications.
The new approach improves both convergence robustness and
 load balancing when FEAST runs on multiple search intervals in parallel.
\end{abstract}


% === AMS === %

\pagestyle{myheadings}
\thispagestyle{plain}
\markboth{S. G\"{U}TTEL, E. POLIZZI, P. TANG, AND G. VIAUD}{ZOLOTAREV QUADRATURE RULES AND LOAD BALANCING FOR FEAST}

\section{Introduction}

The FEAST method \cite{2009_polizzi} is an 
algorithm for computing a few eigenpairs $(\lambda,\mathbf x)$ 
of a large sparse generalized eigenproblem
\begin{equation}\label{eq:gep}
    A \mathbf x = \lambda B \mathbf x,
\end{equation}
where $A\in\mathbb{C}^{N\times N}$ is Hermitian and $B\in\mathbb{C}^{N\times N}$
is Hermitian positive definite. 
%Note that \eqref{eq:gep} is mathematically equivalent to a standard eigenvalue problem with the matrix $M = B^{-1} A$.
This method belongs to the class of contour-based eigensolvers which  have attracted much attention over the past decade. 
Contour-based methods utilize integrals of the form
\begin{equation}\label{eq:moments}
	C_j := \frac{1}{2\pi i} \int_\Gamma \gamma^j (\gamma B - A)^{-1}B \d \gamma = 
	 \frac{1}{2\pi i} \int_\Gamma \gamma^j (\gamma I - M)^{-1} \d \gamma,
\end{equation}
where $M = B^{-1}A$ and $\Gamma$ is a contour in the complex plane enclosing the wanted eigenvalues of $(A,B)$. Typically a quadrature rule is then applied to evaluate this contour integral numerically. 

Probably the first practical method which combined contour integrals and quadrature was presented by Delves and Lyness  
 \cite{1967_delves}, although this was for the (related) purpose of  finding roots of scalar analytic functions (see also \cite{austin2013} for an overview of various methods for this purpose). 
The method presented by Sakurai and Sugiura in \cite{2003_sakurai} (see also \cite{2008_ikegami,2008_sakurai}) makes use of the moments $\mu_j = \mathbf{u}^* C_j \mathbf{v}$
for solving \eqref{eq:gep}.
%Sakurai was the first to propose his projection method consisting in the development
%of a projected Laurent series-type decomposition of $(zB-A)^{-1}$ in 2003 
%and a parallel version in 2008 \cite{2008_sakurai}.
%It was then combined to the Algebraic Sub-structuring method in 2010 for the optimization of the
%location and size of the contours used \cite{2010_sakurai}.
%The results obtained for the parallel version compared well to the Lanczos method \cite{2012_sakurai}.
 This is done by constructing a matrix pencil of small size whose eigenvalues 
correspond to the targeted ones of the original system. The procedure terminates after the reduced system is constructed and its eigenvalues are  obtained. In this sense the method in \cite{2003_sakurai}, sometimes referred to as \emph{SS method}, is non-iterative in nature. The SS method based on explicit moments may become numerically unstable and the so-called \emph{CIRR method} \cite{sakurai2007cirr}  tries to address this problem by using explicit Rayleigh--Ritz projections for Hermitian eigenproblems. A block-version of CIRR applicable to non-Hermitian eigenproblems was presented in  \cite{ikegami2010contour}.

Expressed in terms of moments, FEAST uses only the zeroth moment matrix $C_0$, which
corresponds to the spectral projector onto the invariant subspace associated with the  eigenvalues 
 enclosed by $\Gamma$. Since this projector can  be computed only approximately, FEAST must be an iterative algorithm: it applies an approximate spectral projector 
repeatedly, progressively steering the search space into the direction of an invariant subspace containing the wanted eigenvectors. The original paper \cite{2009_polizzi} demonstrated the effectiveness
of the approach without analysis of convergence, % or numerical issues, %the latter of which
%are reported by Kr\"{a}mer et al. \cite{2013_kraemer}. 
%a detailed convergence and numerical analysis on FEAST 
which was then completed only very recently in \cite{Tang13}.

Consistent with \cite{Tang13},
we use the fact that the FEAST method is equivalent to subspace iteration with implicit orthogonalization 
applied with a rational matrix function $r_m(M)$.
In the original FEAST derivation \cite{2009_polizzi}, the rational function $r_m(z)$
was obtained via quadrature approximation of an indicator function $f(z)$ represented as 
\begin{equation}\label{eq:cauchy}
    f(z) = \frac{1}{2\pi i} \int_\Gamma \frac{\d \gamma}{\gamma - z},
\end{equation}
where $\Gamma$ is a contour enclosing the wanted eigenvalues of $M$.
We will show that the convergence of FEAST is governed by the separation of the wanted and unwanted eigenvalues of $r_m(M)$, and that this separation is determined by the accuracy of the quadrature approximation $r_m$ for $f$. We then use this argument to motivate our new choice of $r_m$, which is not based on contour integration but on a  rational approximant constructed by Zolotarev.

Zolotarev's rational functions are ubiquitous in the design of electronic filters (see, e.g., \cite{blinchikoff1976filtering,van1982analog})  and in this context often referred to as \emph{elliptic filters} or \emph{Cauer filters}. 
Examples from numerical analysis where these functions have proven useful are 
the choice of optimal parameters in the ADI method \cite{Wac88},
 the construction of optimal finite-difference grids \cite{ingerman2000optimal},
in parameter selection problems with rational Krylov methods for matrix functions \cite{guttel2013rational}, or for the optimization of time steps in the Crank--Nicolson method \cite{ML05}, see also \cite{Tod84}.
The use of Zolotarev rational functions (or equivalently, elliptic filters) in the context of FEAST is very natural but
does not seem to have been considered before in the literature, with the exception
of the master thesis  \cite{Via12}.

The outline of this paper is as follows. In Section~\ref{sec:feast}
we briefly review the FEAST method and its connection with
subspace iteration.
In Section~\ref{sec:quad} we compare two different quadrature approaches
that are commonly used, namely the approach 
 based on mapped Gauss quadrature as proposed by Polizzi \cite{2009_polizzi},
and another one based on the trapezoid rule which is close in spirit to that of 
Sakurai and coauthors (see, e.g., \cite{2003_sakurai,ikegami2010contour}).
We also derive a relation between the rational functions obtained from the trapezoid rule on ellipsoidal contours and 
so-called \emph{type-1 Chebyshev filters}. 
While the trapezoid rule seems most natural, Gauss quadrature
turns out to be advantageous if the wanted and unwanted eigenvalues of $M$ are not well separated.
In Section~\ref{sec:zolo} we derive an improved quadrature rule based
on the optimal Zolotarev approximation to the sign function, and compare it in Section~\ref{sec:comp}
with the previous quadrature rules. 
In Section~\ref{sec:load} we discuss the implications of the Zolotarev approach on the load balancing problem which arises when  FEAST runs on multiple search intervals synchronously. 
 We end with Section~\ref{sec:numex} which demonstrates the improvements with numerical experiments.

%Finally, Section~\ref{sec:numex} gives two numerical experiments that
%demonstrate the impact of the improved quadrature rule.

%
%\paragraph{Three other ideas/things to discuss.}
%
%\begin{itemize}
% \item How is summation of the projection matrix done in FEAST? Can we split this and project on each processor for a reduced communication variant of FEAST?
%
% \item Can we change the quadrature weights adaptively using spectral information from Ritz values after each FEAST iteration.
%
% \item Is it possible to use exterior conformal maps to derive quasi-optimal quadrature rules for polygonal domains?
%
%\end{itemize}
%

\section{The FEAST method}\label{sec:feast}

In this section we will explain how the FEAST method is
mathematically equivalent to subspace iteration applied 
with a rational matrix function $r_m(M)$, where $M=B^{-1}A$.
Let  $M = X \Lambda X^{-1}$  be an eigendecomposition of $M$, where
$\Lambda \in \mathbb{R}^{N\times N}$ is a diagonal matrix
whose real diagonal entries are the eigenvalues of $M$ and
the columns of $X \in \mathbb{C}^{N\times N}$ correspond
to the eigenvectors, chosen to be $B$-orthonormal, i.e.,
$X^* B X = I$.
Here is a step-by-step listing of the FEAST method:

\medskip

\begin{enumerate}
 \item Choose $n<N$ random columns of
 $Y_0 := [ \mathbf{y}_1,\ldots,\mathbf{y}_n ]\in\mathbb{C}^{N\times n}$.
 \item Set $k:=1$.
 \item Compute $Z_k := r_m(M) Y_{k-1} \in\mathbb{C}^{N\times n}$.
 \item Compute $\widehat A_k := Z_k^* A Z_k$ and $\widehat B_k := Z_k^* B Z_k$.
 \item Compute a $\widehat B_k$-orthonormal matrix $W_k\in \mathbb{C}^{n\times n}$ and the 
diagonal matrix $D_k=\diag(\vartheta_1,\ldots,\vartheta_n)$ such that
  $\widehat A_k W_k = \widehat B_k W_k D_k.$
 \item Set $Y_k := Z_k W_k$.
 \item If $Y_k$ has not converged, set $k:=k+1$ and goto Step~3.
\end{enumerate}

\medskip

For the rational matrix function $r_m(M)$ in Step~3 to be well-defined we assume
here and in the following that none of the poles of $r_m$ coincides with an eigenvalue of~$M$. 
When $r_m$ has a partial
fraction expansion
\[
	r_m(z) = \sum_{j=1}^{2m} \frac{w_j}{z_j - z},
\]
then Step~3 amounts to the solution of $2m$ decoupled linear systems which
can be solved in parallel (with an appropriate choice of $r_m$ only $m$ linear systems need to be solved
in some cases, see Section~\ref{sec:quad}):
\[
Z_k = r_m(M) Y_{k-1} = \sum_{j=1}^{2m} w_j (z_j B - A)^{-1}(B Y_{k-1}).
\]
In the original formulation of FEAST in \cite{2009_polizzi} the $B$-factor in
$(BY_{k-1})$ is not applied in
Step~3 but at the end of each loop. This  makes a difference only in the
first iteration.

Note that the columns of $Y_k$ for $k\geq 1$ are
$B$-orthogonal because the eigenvector matrix $W_k$ of the reduced pencil
$(\widehat A_k,\widehat B_k)$ computed in Step~5 is $\widehat B_k$-orthonormal and
\[
 Y_k^* B Y_k = W_k^* Z_k^* B Z_k W_k = W_k^* \widehat B_k W_k = I_n.
\]
The orthogonalization procedure is therefore implicitly built into the
Rayleigh--Ritz extraction procedure. FEAST can hence be viewed and analyzed as a subspace iteration with implicit orthogonalization run
with the matrix $r_m(M)$; see, e.g., \cite[\S5.2]{Saa92a}.

The following results are adopted from \cite{Via12} and \cite{Tang13}. We include them for completeness and to motivate our derivations in the following sections.
Let the eigenpairs $(\lambda_j,\mathbf{x}_j)$ of $M$ be ordered such that
\begin{equation}\label{eq:order}
 |r_m(\lambda_1)| \geq |r_m(\lambda_2)| \geq \cdots
 \geq |r_m(\lambda_N)|.
\end{equation}
We introduce the following notations. For any integer $n$, $1 \le n < N$:
\[
\begin{array}{l l l l l}
\TSec{X}{n} & = & [\mathbf{x}_1, \mathbf{x}_2, \ldots, \mathbf{x}_n]         &
                  \in & \mathbb{C}^{N\times n}, \\
\BSec{X}{n} & = & [\mathbf{x}_{n+1}, \mathbf{x}_{n+2}, \ldots, \mathbf{x}_N] &
                  \in & \mathbb{C}^{N\times (N-n)}, \\
\TSec{\Lambda}{n} & = & \diag(\lambda_1, \lambda_2, \ldots, \lambda_n) &
                             \in & \mathbb{R}^{n \times n}, \\
\BSec{\Lambda}{n} & = & \diag(\lambda_{n+1}, \lambda_{n+2}, \ldots, \lambda_N) &
                             \in & \mathbb{R}^{(N-n) \times (N-n)}.
\end{array}
\]
With these notations, the matrix $\TSec{X}{n} \sTSec{X}{n} B$ corresponds to the $B$-orthogonal projector
onto $\spann(\{\mathbf x_1, \mathbf  x_2, \ldots, \mathbf x_n\})$, and likewise
$\BSec{X}{n}\sBSec{X}{n}B$ is the $B$-orthogonal projector onto
$\spann(\{\mathbf  x_{n+1}, \mathbf  x_{n+2}, \ldots, \mathbf  x_N\})$, for any $1 \le n < N$.
In particular, since $X^* B X = I$ implies
$X^{-1} = X^* B$, the eigendecomposition of $M$ and $r_m(M)$ can be written as
\[
M = \TSec{X}{n} \TSec{\Lambda}{n} \sTSec{X}{n} B +
    \BSec{X}{n} \BSec{\Lambda}{n} \sBSec{X}{n} B,
\]
and
\[
r_m(M) = \TSec{X}{n} \, r_m(\Lambda_n) \, \sTSec{X}{n} B +
    \BSec{X}{n}  \,r_m(\Lambda_{n'})\, \sBSec{X}{n} B,
\]
for any $1 \le n < N$.
Note also that
$\TSec{X}{n}\sTSec{X}{n}B + \BSec{X}{n}\sBSec{X}{n}B = I$. The following lemma provides a characterization of $\spann(Z_k)$.

\begin{lemma}
\label{lemma:span}
Consider the FEAST method as described in Steps 1--7 previously. Suppose
$|r_m(\lambda_n)| > 0$, and that
$Y_0$ in Step 1 is such that the $n \times n$ matrix
$\sTSec{X}{n} B Y_0$ is invertible.
Then the matrices $Z_k$ of Step~3 always maintain full column rank $n$ and
\[
\spann(Z_k) = \spann(r_m^k(M) Y_0)
\]
for all iterations $k\geq 1$.
\end{lemma}
\begin{proof}
We will first use an induction argument to show that the matrices $Z_k$
have full column rank and that the matrices $W_k$ are invertible. Suppose
$\sTSec{X}{n} B Y_{k-1}$ is invertible for some $k\geq 1$. Then the
$n\times n$ matrix
\begin{eqnarray*}
\sTSec{X}{n}B Z_k  &=&
\sTSec{X}{n}B\left[ \TSec{X}{n} \, r_m(\Lambda_n) \, \sTSec{X}{n} B +
                    \BSec{X}{n} \, r_m(\Lambda_{n'}) \, \sBSec{X}{n} B
             \right] \, Y_{k-1} \\
 &=&  r_m(\Lambda_n)  \,\sTSec{X}{n} B Y_{k-1}
\end{eqnarray*}
is invertible because
$r_m(\Lambda_n) = \diag(r_m(\lambda_1),\ldots,r_m(\lambda_n))$ is
invertible by the assumption $|r_m(\lambda_n)| > 0$.
In particular $Z_k$ has full column rank. This means that
$\widehat B_k = Z^*_k B Z_k$ is positive definite, resulting in an invertible matrix
$W_k$. Hence $\sTSec{X}{n} B Y_k = (\sTSec{X}{n} B Z_k) W_k$ is also invertible.
By assumption, $\sTSec{X}{n}B Y_0$ is invertible and hence  by induction
$Z_k$ has full column rank and $W_k$ is invertible for $k\geq 1$.

Finally, it is easy to see that $Z_1 = r_m(M) Y_0$, and that
\[
Z_k = r_m^k(M) \, Y_0 \, W_1 \, W_2 \, \cdots \, W_{k-1} \quad
\hbox{for $k\ge 2$}.
\]
Consequently, $\spann(Z_k) = \spann(r_m^k(M)Y_0)$ for $k\geq 1$ as claimed.
\end{proof}

\smallskip

The following theorem is a straightforward adaptation of \cite[Thm.~5.2]{Saa92a} (see \cite{bathe77} for the original result).
Here the $B$-norm of a vector $\mathbf w \in \mathbb{C}^N$ is defined in the usual way as $\| \mathbf w \|_B = (\mathbf w^* B \mathbf w)^{1/2}$.

\begin{theorem}
\label{thm:subspace-iteration}
Consider the FEAST method as described in Steps 1--7 previously. Suppose that 
$|r_m(\lambda_n)| > 0$ and 
$Y_0$ in Step~1 is such that the $n \times n$ matrix $\sTSec{X}{n} B Y_0$ is invertible.
Let $P_k$ be the $B$-orthogonal projector onto the subspace $\spann(Z_k)$. Then for
each $j=1,2,\ldots,n$ there is a constant $\alpha_j$ such that
\[
\|(I-P_k)\mathbf{x}_j\|_B \leq \alpha_j \left| \frac{r_m(\lambda_{n+1})}{r_m(\lambda_j)} \right|^k
\]
for iterations $k\geq 1$, where
$(\lambda_j,\mathbf{x}_j)$ is the $j$-th eigenpair
of $M$ with the ordering \eqref{eq:order}.
In particular, $\|(I-P_k)\mathbf{x}_j\|_B \rightarrow 0$ as long as
$|r_m(\lambda_j)| > |r_m(\lambda_{n+1})|$.
\end{theorem}
\begin{proof}
As observed previously,
$I = \TSec{X}{n}\sTSec{X}{n}B +
     \BSec{X}{n}\sBSec{X}{n}B$. Therefore
\begin{eqnarray*}
Y_0  & = & (\TSec{X}{n}\sTSec{X}{n}B + \BSec{X}{n}\sBSec{X}{n}B) \,Y_0  \\
     & = & (\TSec{X}{n} + \BSec{X}{n} (\sBSec{X}{n} B Y_0)
                                      (\sTSec{X}{n} B Y_0)^{-1})\,
           \sTSec{X}{n} B Y_0.
\end{eqnarray*}
Hence $\spann(Y_0) = \spann( \TSec{X}{n} + \BSec{X}{n} W )$,
where $W$ is the $(N-n)\times n$ matrix
\[
W = (\sBSec{X}{n} B Y_0) (\sTSec{X}{n} B Y_0)^{-1}.
\]
Writing $W$ as $[\mathbf{w}_1, \mathbf{w}_2, \ldots, \mathbf{w}_n]$, 
the vector $\mathbf{x}_j + \BSec{X}{n}\,\mathbf{w}_j$ is an element of
$\spann(Y_0)$. Define the constant
$\alpha_j$ as $\|\mathbf{w}_j\|_2$. By Lemma~\ref{lemma:span},
$\spann(Z_k) = \spann(r_m^k(M)Y_0)$ and thus
$
r_m^k(M)(\mathbf{x}_j + \BSec{X}{n}  \mathbf{w}_j) \in
\spann(Z_k).
$
But $r_m^k(M) = X \, r_m^k(\Lambda) \, X^{-1}$ and thus
\[
r_m^k(M)(\mathbf{x}_j + \BSec{X}{n}  \mathbf{w}_j)
= r_m^k(\lambda_j)\, \mathbf{x}_j +
  \BSec{X}{n} \, r_m^k\left(\BSec{\Lambda}{n}\right) \mathbf{w}_j.
\]
Therefore, the vector $\mathbf{x}_j + \BSec{X}{n} \widetilde{\mathbf{w}}_j$ is an
element of $\spann(Z_k)$,
where
\[
\widetilde{\mathbf{w}}_j =
\diag\left( \frac{ r_m^k(\lambda_{n+1})}{r_m^k(\lambda_j)},
       \frac{ r_m^k(\lambda_{n+2})}{r_m^k(\lambda_j)}, \ldots,
       \frac{ r_m^k(\lambda_N)}{r_m^k(\lambda_j)} \right)
 \,\mathbf{w}_j.
\]
Hence
$
\| \widetilde{\mathbf{w}}_j \|_2 \le
\alpha_j \, \left| r_m(\lambda_{n+1}) / r_m(\lambda_j) \right|^k
$.
Therefore, inside $\spann(Z_k)$ lies a vector
$\mathbf{x}_j + \mathbf{e}_j$ with
$
\| \mathbf{e}_j \|_B = \| \widetilde{\mathbf{w}}_j \|_2
\le
\alpha_j \, \left| r_m(\lambda_{n+1}) / r_m(\lambda_j) \right|^k .
$
Finally,
\begin{eqnarray*}
\| (I-P_k) \mathbf{x}_j \|_B
 & = &
 \min_{\mathbf{z} \in \spann(Z_k)}
 \| \mathbf{x}_j - \mathbf{z} \|_B \\
 & \le &
 \| \mathbf{e}_j \|_B \\
 & \le &
\alpha_j \, \left| \frac{r_m(\lambda_{n+1}) }{ r_m(\lambda_j) } \right|^k ,
\end{eqnarray*}
which completes the proof.
\end{proof}

We learn from Theorem~\ref{thm:subspace-iteration}  that 
fast convergence can be achieved for a wanted eigenpair 
$(\lambda_j,\mathbf x_j)$   if the 
ratio $|r_m(\lambda_{n+1})/r_m(\lambda_j)|$ is small ($j\leq n$). 
This is an  approximation problem which 
we will investigate closer in the following sections.

\section{Two simple quadrature rules}\label{sec:quad}

We assume without loss of generality that the pencil $(A,B)$ has been transformed linearly to $(\alpha A - \beta B, B)$
such that the wanted  eigenvalues are contained in the interval $(-1,1)$.
For a given scaling parameter $S>1$ we  define a family of ellipses $\Gamma_S$ as 
\begin{equation}\label{eq:Gamma}
    \Gamma_S = \left\{ \gamma : \gamma = \gamma(\theta) = \frac{S e^{i\theta} + S^{-1}e^{-i\theta}}{S + S^{-1}}, \ \theta\in [0,2\pi) \right\}.
\end{equation}
Note that $\gamma(\theta) = \cos(\theta) + i \frac{S-S^{-1}}{S+S^{-1}}\sin(\theta)$, 
hence these ellipses enclose the interval $(-1,1)$ and pass through the interval endpoints $\pm 1$.
As $S\to\infty$, the ellipses approach the unit circle.
After a straightforward change of variables one can evaluate the integral
\eqref{eq:cauchy} with contour $\Gamma=\Gamma_S$ via integration over $[0,2\pi]$ as
\begin{eqnarray}\label{eq:fz}
    f(z) &=\,& \frac{1}{2\pi i} \int_0^{2\pi} \frac{\gamma'(\theta)}{\gamma(\theta) - z}\d \theta
    =:   \int_0^{2\pi} g_z(\theta) \d \theta,
\end{eqnarray}
where
\[
	g_z(\theta) := \frac{1}{2\pi}   \frac{ (S e^{i\theta} - S^{-1}e^{-i\theta})/(S + S^{-1}) }
    {(S e^{i\theta} + S^{-1}e^{-i\theta})/(S + S^{-1}) - z}.
\]
Two different approaches for the numerical approximation of the integral \eqref{eq:fz} 
 have been considered in the context of contour-based eigensolvers.

\subsection{Gauss quadrature}

It was proposed in \cite{2009_polizzi} to use $m$ Gauss quadrature nodes $\theta_j^{(G)}$ ($j=1,\ldots,m$) on the interval $[0,\pi]$,  and another set of $m$ Gauss quadrature nodes $\theta_j^{(G)}$ ($j=m+1,\ldots,2m$) on the interval $[\pi,2\pi]$. Denoting the corresponding Gauss weights by $\omega_j^{(G)}$, we have for \eqref{eq:fz} the quadrature approximation
 \[
    f(z) \approx  \sum_{j=1}^{2m} \omega_j^{(G)} g_z(\theta_j^{(G)})
    =:  r_m^{(G)}(z),
 \]
with a rational function $r_m^{(G)}$.
Defining the mapped Gauss nodes and weights
\[
     z_j^{(G)} = \frac{S e^{i\theta_j^{(G)}} + S^{-1}e^{-i\theta_j^{(G)}}}{S + S^{-1}},
    \quad w_j^{(G)} = \frac{\omega_j}{2\pi}  \frac{S e^{i \theta_j^{(G)}} - S^{-1}e^{-i \theta_j^{(G)}}}{S + S^{-1}}, \quad j=1,\ldots,2m,
\]
the function $r_m^{(G)}$ can be written in the form
\begin{equation}\label{eq:gaussrule}
	r_m^{(G)}(z) = \sum_{j=1}^{2m} \frac{w_j^{(G)}}{z_j^{(G)} - z}.
\end{equation}
This is a rational function of type $(2m-1,2m)$. By construction, its $2m$ poles have a four-fold symmetry about the origin,
\[
    z_j^{(G)} = -\overline{z_{m+1-j}^{(G)}}
    = - z_{m+j}^{(G)} = \overline{z_{2m+1-j}^{(G)}} \quad \text{for\ } j=1,\ldots,m,
\]
in particular, the poles occur in complex conjugate pairs.
If $A$ and $B$ are real symmetric this can be computationally convenient because for a real vector~$\mathbf v$ one has
$$
\overline{(z B - A)^{-1} \mathbf v} = (\overline{z} B - A)^{-1}\mathbf v,
$$
and hence the number of linear systems to be solved for computing $r_m^{(G)}(B^{-1}A)\mathbf v$ is only $m$ instead of $2m$.

A graphical illustration of $r_m^{(G)}$ is given in Figure~\ref{fig:gauss}. When $S$ decreases this rational function  becomes quite ``wiggly'' on the interval $(-1,1)$, with the oscillations being caused by the nearby poles and hence becoming larger as the ellipse gets flatter. 

%On the other, the transition from the value $\approx 1$ in $(-1,1)$ down to $\approx 0$ outside this interval gets sharper when the ellipse gets flatter.

% \begin{figure}
%   % Requires \usepackage{graphicx}
%   \includegraphics[width=1.0\linewidth]{fig/gauss}\\[-5mm]
%   \caption{Rational function $r_m^{(G)}$ obtained from the mapped Gauss quadrature rule.
%   The poles lie on the contour $\Gamma_S$ shown on the left and the graph is shown on the
%   right. Here, $S=2$ and $m=3$. This rational function can be seen as an approximation for an
%   indicator function.}\label{fig:gauss}
% \end{figure}

\begin{figure}
  \centering\includegraphics[width=.7\linewidth]{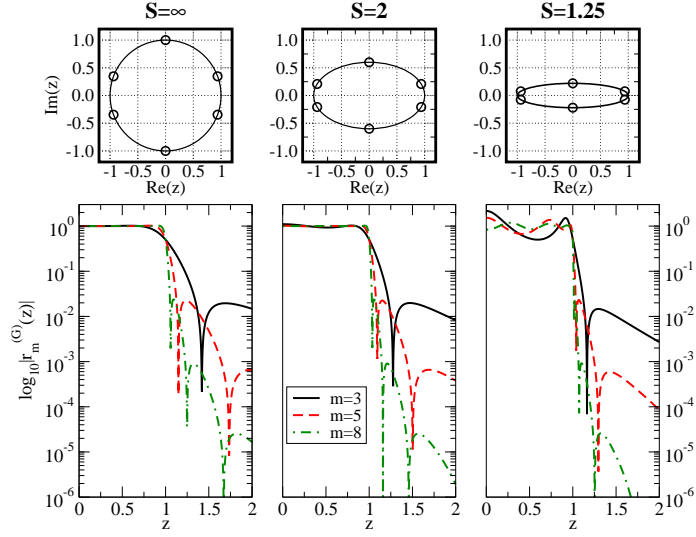}\\
  \caption{
The contours $\Gamma_S$ along with the associated rational functions $r_m^{(G)}$ obtained from the mapped Gauss quadrature rule for three different values of $S\in\{\infty,2,1.25\}$. The modulus of $r_m^{(G)}$ is plotted for each value of $S$ and for three different values of $m\in\{ 3,5,8\}$ over the interval $[0,2]$ (it is a symmetric function).  For clarity, the poles of $r_m^{(G)}$ are shown for the case $m=3$ only.}
\label{fig:gauss}
\end{figure}

\subsection{Trapezoid rule}

As the integrand $g_z$ in \eqref{eq:fz} is a $2\pi$-periodic function,
it appears most natural to use the trapezoid rule for its integration over $[0,2\pi]$. Indeed this is the preferred choice of quadrature rule in the moment-based methods (see, e.g., \cite{2003_sakurai,ikegami2010contour}). We refer to  \cite{trefethen2013exponentially} for a review of the trapezoid rule and its properties.

Let us take equispaced quadrature nodes $\theta_j^{(T)} = \pi (j-1/2)/m$ and equal weights
$\omega_j^{(T)} = \pi/m$, $j=1,\ldots,2m$, and use for \eqref{eq:fz} the trapezoid  approximation
\[
    f(z) \approx  \sum_{j=1}^{2m} \omega_j^{(T)} g_z(\theta_j^{(T)})
    =:  r_m^{(T)}(z),
\]
with a rational function $r_m^{(T)}$ of type at most $(2m-1,2m)$.
Defining the mapped trapezoid nodes and weights
\[
     z_j^{(T)} = \frac{S e^{i\theta_j^{(T)}} + S^{-1}e^{-i\theta_j^{(T)}}}{S + S^{-1}},
    \quad w_j^{(T)} = \frac{1}{2m}  \frac{S e^{i \theta_j^{(T)}} - S^{-1}e^{-i \theta_j^{(T)}}}{S + S^{-1}}, \quad j=1,\ldots,2m,
\]
the rational function $r_m^{(T)}$ can be written in the form
\begin{equation}\label{eq:traprule}
	r_m^{(T)}(z) = \sum_{j=1}^{2m} \frac{w_j^{(T)}}{z_j^{(T)} - z}.
\end{equation}
We now show that $r_m^{(T)}$ has a close connection with Chebyshev polynomials. 

\begin{lemma}\label{lem:trapez}
The rational function $r_m^{(T)}$ can be written as fractional transformations of $T_{j}(z)=\cos(j\arccos(z))$, the first-kind Chebyshev polynomial of degree $j$. More precisely,
\begin{equation}
r_m^{(T)}(z) = \frac{1}{\alpha + \beta \: T_{2m}(\frac{S+S^{-1}}{2} z)} = 
\frac{1}{(\alpha - \beta) + 2 \beta \: T_{m}(\frac{S+S^{-1}}{2} z)^2}
\label{eq1}
\end{equation}
with
\begin{equation}\label{eq:alphbet}
\begin{tabular}{ll}
$\displaystyle \alpha = \frac{ S^{2m} + S^{-2m} }{ S^{2m} - S^{-2m} }$, & $\displaystyle \beta = \frac{2}{S^{2m}-S^{-2m}}$.
\end{tabular}
\end{equation}
Therefore, $r_m^{(T)}$ is of exact type $(0,2m)$. Moreover, it is equioscillating $2m+1$ times on the interval $[-2/(S+S^{-1}),2/(S+S^{-1})]$, alternating between the values $(\alpha \pm \beta)^{-1}$.
\end{lemma}
\begin{proof}
We only consider the first equality in \eqref{eq1}, the second following from the relation
$T_{2m}(z) = 2 T_{m}(z)^2 - 1$. 
Let the rational function
\begin{equation}
w(z) = \frac{1}{\alpha + \beta \: T_{2m}(\frac{S+S^{-1}}{2} z)}
\end{equation}
be defined with $\alpha$ and $\beta$ as in \eqref{eq:alphbet}. Being clearly of type $(0,2m)$, it suffices to prove that $w$ has the same poles as $r_m^{(T)}$ defined in \eqref{eq:traprule} and the same residues at these points. We make use of the following formulas for Chebyshev polynomials of a complex variable \cite{mason}, namely
\begin{equation}\nonumber
\begin{array}{lll}
\displaystyle T_{2m} \left(\frac{z + z^{-1}}{2}\right) &=& \displaystyle \frac{z^{2m} + z^{-2m}}{2}, \\
\displaystyle U_{2m-1}\left(\frac{z+z^{-1}}{2}\right) &=& \displaystyle \frac{z^{2m} - z^{-2m}}{z-z^{-1}}, \\
\displaystyle  \frac{\d}{\!\d z} T_{2m}(z) &=& \displaystyle 2m \: U_{2m-1}(z),
\end{array}
\end{equation}
where $U_{2m-1}$ is the Chebyshev polynomial of the second kind of degree $2m-1$. Defining $u_j = S e^{i\theta_j^{(T)}}$, then
\begin{eqnarray*}
\frac{S+S^{-1}}{2} z_j^{(T)} = \frac{u_j + u_j^{-1}}{2}  \quad \text{and} \quad u_j^{2m} = S^{2m} e^{2i \cdot m\theta_j^{(T)}} = -S^{2m}.
\end{eqnarray*}
Therefore,
\begin{eqnarray*}
\alpha + \beta \: T_{2m}\left(\frac{S+S^{-1}}{2} z_j^{(T)}\right)
& = & \alpha + \beta \: T_{2m}\left(\frac{u_j + u_j^{-1}}{2}\right) \\
& = & \alpha + \beta \frac{u_j^{2m} + u_j^{-2m}}{2} \\
& = & \alpha - \beta \frac{S^{2m} + S^{-2m}}{2},
\end{eqnarray*}
which gives zero when inserting the values of $\alpha$ and $\beta$. It remains to show that the residue of $w$ at $z_j^{(T)}$ is precisely $-w_j^{(T)}$. To this end we make use of the fact that the residue at a point $z$ of a rational function $p/q$, where $p$ and $q$ are polynomials such that $q$ has a simple root at $z$ and $p$ is nonzero there, is given by $p(z)/q'(z)$. The residue of $w$ at $z=z_j^{(T)}$ is hence given by
\begin{eqnarray*}
 \left.
 \frac{1}{{\frac{\d}{\d z}} \left(\alpha + \beta \:  T_{2m}\big(\frac{S+S^{-1}}{2} z\big)\right)  }
 \right|_{z=z_j^{(T)}} 
  &=& \frac{1}{\beta\, m\, (S+S^{-1})\, U_{2m-1}\big(\frac{S+S^{-1}}{2} z_j^{(T)}\big)} \\
  &=& \frac{1}{\beta\, m\, (S+S^{-1})\, U_{2m-1}\big(\frac{u_j + u_j^{-1}}{2}\big)} \\
  &=& \frac{1}{\beta\, m\, (S+S^{-1})\, \frac{u_j^{2m} - u_j^{-2m}}{u_j - u_j^{-1}}} \\
  &=& \frac{u_j - u_j^{-1}}{\beta\, m\, (S+S^{-1})\, (-S^{2m} + S^{-2m})},
\end{eqnarray*}
which indeed agrees with $-w_j^{(T)}$. The equioscillation property of $r_m^{(T)}$ follows directly from the equioscillation of $T_{2m}$.
\end{proof}

We learn from Lemma~\ref{lem:trapez} that $r_m^{(T)}$ is precisely a \emph{type-1 Chebyshev filter} as commonly used in electronic filter design; see, e.g., \cite[\S~13.5]{hamming1989digital}.  
A graphical illustration of $r_m^{(T)}$ is given in Figure~\ref{fig:trapez}. Note that this rational function is perfectly equioscillating on the interval $[-2/(S+S^{-1}),2/(S+S^{-1})]$, which becomes wider as the ellipse gets flatter ($S\to 1$). On the other hand, the function  values are between  $(\alpha\pm\beta)^{-1}$ with $\beta = 2/(S^{2m} - S^{-2m})$, so the oscillations become larger as $S\to 1$. 
In the other limiting case, when $S\to\infty$, there are no oscillations and
\begin{equation}\label{eq:rmTinf}
	r_m^{(T)}(z) = \frac{1}{2m} \sum_{j=1}^{2m} \frac{e^{i\pi (j-1/2)/m}}{e^{i\pi (j-1/2)/m} - z},
\end{equation}
which is also known as the \emph{Butterworth filter}; see, e.g., \cite[\S~12.6]{hamming1989digital} or \cite{austin2013}. This relation between the type-1 Chebyshev and Butterworth filters is well known in the literature (see, e.g., \cite[p.~119]{blinchikoff1976filtering}).
By symmetry considerations one can show that $r_m^{(T)}$ in \eqref{eq:rmTinf} attains the value $1/2$ for $z=\pm 1$. 
This property is also shared by $r_m^{(G)}$ as we show in the following remark.

% 
% However, the decay of $r_m^{(T)}$ away from the interval $[-1,1]$ becomes faster as $S\to 1$. 
% The following lemma puts this decay in relation to the growth of the oscillations and will be helpful later to further assess the performance of $r_m^{(T)}$ when used in FEAST.
% 
% \begin{lemma}
%  Let $G$ be a fixed number such that $0<G<1$, and let $m\geq 2$. Then ... (show monotonicity of the max/min quotient in $S$).
% \end{lemma}
% \begin{proof}
%  First of all, it is clear from Lemma~\ref{lem:trapez} that $r_m^{(T)}(z)>0$ for all $z\in\mathbb{R}$, and that it decays strictly monotonically as one moves away from the interval $[-2/(S+S^{-1}),2/(S+S^{-1})]$ of equioscillation. 
% \end{proof}
%\SG{Mention that $r_m^{(T)}(-1)=r_m^{(T)}(1)$ is not necessarily $1/2$ (at least when $\Gamma$ is not a circle! Check the corresponding property for the Gauss rule on an ellipse.}

% \begin{figure}
%   % Requires \usepackage{graphicx}
%   \includegraphics[width=1.0\linewidth]{fig/trapez}\\[-5mm]
%   \caption{Rational function $r_m^{(T)}$ obtained from the mapped trapezoidal quadrature rule.
%   The poles lie on the contour $\Gamma_S$ shown on the left and the graph is shown on the
%   right. Here, $S=2$ and $m=3$. Again, this rational function can be seen as an approximation for an
%   indicator function.}\label{fig:trapez}
% \end{figure}

\begin{figure}
  \centering\includegraphics[width=.7\linewidth]{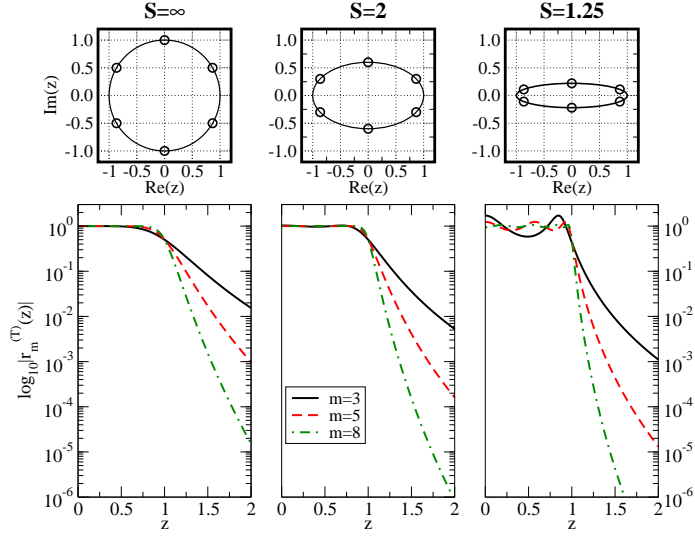}\\
  \caption{
The contours $\Gamma_S$ along with the associated rational functions $r_m^{(T)}$ obtained from the mapped trapezoid quadrature rule for three different values of $S\in\{\infty,2,1.25\}$. The modulus of $r_m^{(T)}$ is plotted for each value of $S$ and for three different values of $m\in\{ 3,5,8\}$ over the interval $[0,2]$ (it is a symmetric function).  For clarity, the poles of $r_m^{(T)}$ are shown for the case $m=3$ only.}
\label{fig:trapez}
\end{figure}

\begin{remark}\label{rem:symm}
Assume that the poles and corresponding weights have a four-fold symmetry about the origin, i.e., if $(w,z)$ is a weight--pole pair, then also $(-\overline{w},-\overline{z})$, $(-w,-z)$, and $(\overline{w},\overline{z})$ are weight--pole pairs. In this case one can verify that
\begin{equation*}
r_m(\pm 1) = \sum_{j=1}^{2m} \frac{w_j}{z_j - 1} = \sum_{j=1}^{m/2} \frac{\omega_j}{\pi} \frac{S^8-1}{S^8 - 2 \cos(2 \theta_j) S^4 + 1}.
\end{equation*}
Whatever the values $\theta_j$, we have
\begin{equation*}
\frac{1}{\pi} \frac{S^4-1}{S^4+1} \sum_{j=1}^{m/2} \omega_j < r_m(\pm 1) < \frac{1}{\pi} \frac{S^4+1}{S^4-1} \sum_{j=1}^{m/2} \omega_j.
\end{equation*}
For $S \rightarrow \infty$ we obtain 
$
r_m(\pm 1) = \pi^{-1} \sum_{j=1}^{m/2} \omega_j.
$
Therefore for both the Gauss and the trapezoid quadrature rules we have 
$
\sum_{j=1}^{m/2} \omega_j^{(G),(T)} = \pi/2,
$
hence 
\begin{equation*}
\lim_{S \rightarrow \infty} r_m^{(G),(T)}(\pm 1) = \frac{1}{2}.
\end{equation*}
\end{remark}

\section{A method based on Zolotarev approximants}\label{sec:zolo}

Both quadrature rules in Section~\ref{sec:quad} achieve a small approximation
error for \eqref{eq:cauchy} throughout the complex plane, except when $z$ is close to
the contour $\Gamma$.
Assume again that the wanted eigenvalues of $M$ are contained in the interval $(-1,1)$.
Then for a fast convergence of FEAST, in view of Theorem~\ref{thm:subspace-iteration},
our main concern should be the accuracy of $r_m(M)$ as an approximation
to the indicator function $\ind_{[-G,G]}(M)$, where
\[
 \ind_{[-G,G]}(z) = \begin{cases} 1 &\mbox{if } z \in [-G,G] \\
0 & \mbox{otherwise}, \end{cases}
\]
with some $G<1$. We will refer to $G$ as
the \emph{gap parameter}, because it is related to the gap between the wanted and unwanted eigenvalues. The smaller the value of $G$, the larger the gap.
Since $M$ has real eigenvalues, it seems natural to concentrate all of $r_m$'s
``approximation power'' to the real line. In other words, we are looking
for a rational function $r_m$ of degree $2m$ such that $r_m$ is closest
to $1$ on a largest possible interval $[-G,G]\subset (-1,1)$, and closest to $0$ on
a largest possible subset of the complement.
Such a rational function is explicitly known due to an ingenious construction
of Zolotarev~\cite{Zol77} and in the filter design literature typically referred to as 
\emph{band-pass Cauer filter} or \emph{elliptic filter} (see, e.g., \cite[\S~3.7.4]{blinchikoff1976filtering} 
or \cite[\S~13.6]{van1982analog}). 
The construction of this filter makes use of elliptic functions.

Let the Jacobi elliptic
function $\sn(w;\kappa) = x$ be defined by\footnote{The definition of
elliptic functions is not consistent in the literature. We stick to the definitions used in
\cite[\S 25]{Akh90}. For example, in \textsc{Matlab} one would type
\texttt{sn = ellipj(w,kappa{\textasciicircum}2)} and \texttt{K = ellipke(kappa{\textasciicircum}2)}
to obtain the values of $\sn(w;\kappa)$ and $\K(\kappa)$, respectively.}
\[
w = \int_0^x \frac{1}{\sqrt{(1-t^2)(1-\kappa^2 t^2)}} \d t,
\]
and let the complete elliptic integral for the modulus $\kappa$ be denoted by
\[
    \K(\kappa) = \int_0^1 \frac{1}{\sqrt{(1-t^2)(1-\kappa^2 t^2)}} \d t.
\]
The following well-known theorem summarizes one of Zolotarev's findings;
we use a formulation given by Akhiezer \cite[Chapter~9]{Akh90}.

\begin{theorem}[Zolotarev, 1877]\label{thm:zolo}
The best uniform rational approximant of type $(2m-1,2m)$
for the signum function $\sgn(x)$ on the set
$[-R,-1]\cup [1,R]$, $R>1$, is given by
\[
    s_m(x) = x D \frac{\prod_{j=1}^{m-1} (x^2 + c_{2j})}{\prod_{j=1}^m (x^2 + c_{2j-1})}
    \quad \text{with} \quad c_j = \frac{\sn^2(j\K(\kappa)/(2m);\kappa)}{1-\sn^2(j\K(\kappa)/(2m);\kappa)},
\]
where $\kappa=\sqrt{1-1/R^{2}}$ and the constant $D$ is uniquely
determined by the condition
\[
    \min_{x\in [-R,-1]} s_m(x) + 1 = \max_{x\in [1,R]} -s_m(x) + 1.
\]
\end{theorem}

The last normalization condition in Theorem~\ref{thm:zolo} ensures that
$s_m(x)$ is equioscillating about the value $-1$ on $[-R,-1]$, and
equioscillating about the value $1$ on $[1,R]$. In fact, it is known that
there is a number of $4m+2$ equioscillation points, a number that clearly
has to be even due to the symmetry $s_m(-x)=-s_m(x)$.
This is one equioscillation point more than required by Chebyshev's characterization
theorem for uniform best rational approximations (see, e.g., \cite[\S 2.2]{PP11}), which
states that a rational function of type $(\mu,\nu)$ with $\mu+\nu+2$ equioscillation
points is a unique best approximant\footnote{Note that Chebyshev's classical
equioscillation criterion is typically stated for a single closed interval and does not
strictly apply in the case of two intervals. However, looking closer at Zolotarev's construction \cite{Zol77}
we find that it is based on a weighted best rational  approximant  for $1/\sqrt{x}$ on the single interval $[1,R^2]$, on which the equioscillation criterion holds. Zolotarev then uses the relation $\sgn(x) = x/\sqrt{x^2}$ to find $s_m(x)$.}.

%
%
%\SG{\footnote{\SG{Gautier: I still believe we should avoid stating this result since in its conventional formulation, it holds only on a single closed interval, not a union of closed intervals. It might be true but from what I remember of my last year's readings, the usual proof cannot be adapted for a union of closed intervals.}}}

\begin{figure}
  \centering\includegraphics[width=0.7\linewidth]{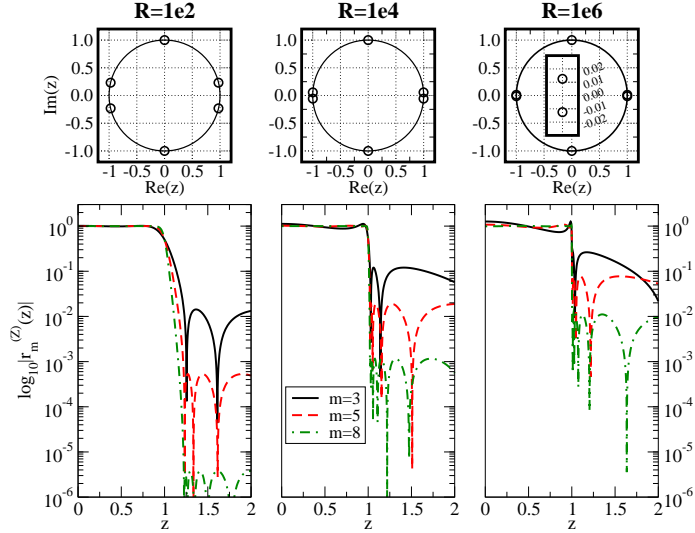}\\
  \caption{Transformed rational function $r_m^{(Z)}(z)$ based on Zolotarev's approximant. 
The parameters are $m=3$ and $R\in\{1e2,1e4,1e6\}$ is varied. The interval of equioscillation about the value $1$ 
is $[-G,G]$, where $G=(\sqrt{R}-1)/(\sqrt{R}+1)$; see formula~\eqref{eq:mzolo}. As stated in Corollary~\ref{cor:zolo}
all poles lie on the unit circle and appear in complex conjugate pairs.}\label{fig:zolotarev}
\end{figure}

%The rational function $s_m(x)$ of Theorem~\ref{thm:zolo} is illustrated
%in the left of Figure~\ref{fig:zolotarev}, for parameters $R=50$ and $m=3$.
%We have highlighted special function values
%of $s_m(x)$ with the red circles:

Let us briefly highlight some properties of  $s_m(x)$. 
First of all, $s_m(0)=0$ due to the symmetry, and $s_m(\infty)=0$ 
as $s_m$ is a rational function
of type $(2m-1,2m)$.
Let us define by $E_m$ the maximal modulus of the error function
$e_m(x) := \sgn(x)-s_m(x)$ over the set $[-R,-1]\cup [1,R]$, i.e.,
\[
 E_m := \max_{x\in [-1,R]\cup [1,-R]} |e_m(x)| = \max_{x\in [-1,-R]\cup [1,R]} |\sgn(x)-s_m(x)|.
\]
Then $|e_m(x)|$ takes on its maximum $E_m$ at the points $x\in\{-R,-1,1,R\}$.  
In \cite[eq.~(3.17)]{ML05} lower and upper bounds on $E_m$ have been given as
\begin{equation}\label{eq:zolerr}
 \frac{4 \rho^m}{1 + \rho^m} \leq E_m \leq  4\rho^m,
\end{equation}
where
$$ \mu  = \left(\frac{\sqrt{R} - 1}{\sqrt{R} + 1}\right)^2, \quad \mu' = \sqrt{1-\mu^2}, \quad
\rho = \rho(\mu) = \exp \left( - \frac{\pi \K(\mu')}{2 \K(\mu)}\right).
$$

Our aim is to determine a Moebius transformation
\begin{equation}\label{eq:rtrans}
    x = t(z) = \frac{a + bz}{c + dz}
\end{equation}
such that the rational function
\begin{equation}\label{eq:rzolo}
 r_m^{(Z)}(z) := \frac{s_m(t(z)) + 1}{2}
\end{equation}
is an approximation of the indicator
function $\ind_{[-G,G]} (z)$ on some interval $[-G,G]\subset (-1,1)$ with gap parameter $G<1$.

Due to the symmetry of the indicator function, it is natural to demand that $r_m^{(Z)}(z) = r_m^{(Z)}(-z)$
for real $z$, and we will also prescribe $r_m^{(Z)}(-1)=r_m^{(Z)}(1)=1/2$. This yields the following
conditions for the transformation $t$:
\[
    t(-1) = 0, \quad t(-G) = 1,  \quad t(G) = R, \quad t(1) = \infty.
\]
From these conditions the transformation $t$ and $G$ are readily determined as
\begin{equation}\label{eq:mzolo}
    x = t(z) = \sqrt{R}\, \frac{1+z}{1-z}, \quad G = \frac{\sqrt{R}-1}{\sqrt{R}+1}.
\end{equation}
By construction, the rational function $r_m^{(Z)}$ is indeed equioscillating about
the value $1$ for $z\in [-G,G]$, and equioscillating about the value $0$ for
$z\in [-\infty,-G^{-1}]$ and $z\in [G^{-1},+\infty]$.
The number of $4m+2$ equioscillation points of $s_m(x)$ is inherited by $r_m^{(Z)}(z)$. 
Note that a rational transformation of type $(1,1)$
 inserted into a rational function of type
$(2m-1,2m)$ in general gives a rational function of type $(2m,2m)$. For a visual example 
see Figure~\ref{fig:zolotarev}.

The following corollary summarizes the above findings.

\begin{corollary}\label{cor:zolo}
The rational function $r_m^{(Z)}$ given by \eqref{eq:rzolo} and \eqref{eq:mzolo}
is the best uniform rational approximant of type $(2m,2m)$ of the indicator function
$\ind_{[-G,G]}(z)$ on $$\left[-G,G\right]
\quad \text{and} \quad [ -\infty,-G^{-1} ] \cup
[G^{-1},+\infty ].$$
The error curve $e_m'(z) := \ind_{[-G,G]}(z) - r_m^{(Z)}(z)$
equioscillates on these sets with error $E_m':= \max_{z\in [-G,G]} |e_m'(z)|$
bounded by
\[
 \frac{2 \rho^m}{1 + \rho^m} \leq E_m' \leq  2\rho^m,
\]
where
$$
\mu  = G^2, \quad \mu' = \sqrt{1-\mu^2}, \quad
\rho = \rho(\mu) = \exp \left( - \frac{\pi \K(\mu')}{2 \K(\mu)}\right).
$$
Moreover, all $2m$ poles of $r_m^{(Z)}$ lie on the unit circle and appear in
complex conjugate pairs.
\end{corollary}

\begin{proof}
The first statement follows from the fact that $r_m^{(z)}$ defined in \eqref{eq:rzolo}
 has been obtained from $s_m(x)$ by the bijective transformation $x=t(z)$ in \eqref{eq:rtrans}.

The error inequalities follow from \eqref{eq:rzolo} and \eqref{eq:zolerr},
and the fact that $\mu = G^2$.

Finally, from \eqref{eq:rzolo}, we find that $z_j$ is a pole of $r_m^{(Z)}$ if and only if $t(z_j)$ is a pole of $s_m$ defined in Theorem~\ref{thm:zolo}. Inverting the relation $x = t(z)$ and using the fact that the poles of $s_m$ are $\pm i \sqrt{c_{2j-1}}$, $j = 1, \dots, m$, the poles of $r_m^{(Z)}$ are found to be
\begin{eqnarray*}
\frac{\pm i \sqrt{c_{2j-1}} - \sqrt{R}}{\pm i \sqrt{c_{2j-1}} + \sqrt{R}} = \frac{c_{2j-1} - R}{c_{2j-1} + R} \pm i \frac{2 \sqrt{c_{2j-1} R}}{c_{2j-1}+R}, \quad j = 1, \dots, m,
\end{eqnarray*}
which are complex conjugate and of modulus 1.
\end{proof}

% \SG{TODO: Short FEAST Matlab code}
% \begin{verbatim}
% Q = randn(N,m);
% for iter = 1:maxiter,
%     Y = absterm*Q;
%     for j = 1:length(poles),
%         Y = Y + weights(j)*((poles(j)*B - A)\(B*Q));
%     end
%     [W,D] = eig(Y'*A*Y,Y'B*Y);
%     Q = Y*W;
%     Q = Q*diag(1./sqrt(sum(abs(Q).^2)));
% end
% \end{verbatim}

\begin{remark}
 When $G$ (and hence $\mu$) is sufficiently close to $1$, it is possible to use  \cite[(17.3.11) and (17.3.26)]{AS84} to derive the asymptotically sharp estimates $K(\mu) \simeq \log(4/\mu')$ and $K(\mu') \simeq \pi/2$ (see also \cite{ML05}), and thereby
give the estimate
\[
 E'_m \simeq 2 \exp\left( -m \frac{\pi^2}{4 \log(4/\mu')}\right) = \exp\left( -m \frac{\pi^2}{2\log(16/(1-G^4))}\right)
\]
in terms of elementary functions.
\end{remark}

\begin{remark}
It may be instructive to study the simplest Zolotarev function $r_m^{(Z)}$ for $m=1$, that is, a
rational function of type $(2m,2m) = (2,2)$.
Let a gap parameter $G<1$ be given.
%The  wanted eigenvalues of $M$
%are assumed to lie in $[-G,G]$,
%and the outer eigenvalues lie in $(-\infty,-G^{-1}]\cup [G^{-1},+\infty)$.
As the poles of $r_m^{(Z)}$ lie on the
unit circle, and due to symmetry must be $\pm i$, this rational function is of the form
\[
    r_1^{(Z)}(z) = \gamma + \frac{2\delta}{z^2 + 1},
\]
with real numbers $\gamma$ and $\delta$.
Due to the equioscillation property we have $r_1^{(Z)}(\infty) = \gamma = 1 - r_1^{(Z)}(0)$, from
which we find that $2\delta = 1- 2\gamma$. Also due to equioscillation we have $r_1^{(Z)}(G) = 1+\gamma$, from which we then find $\gamma = - G^2 / 2$, i.e.,
\[
    r_1^{(Z)}(z) = -\frac{G^2}{2} + \frac{1 + G^2}{z^2 + 1} =
    -\frac{G^2}{2} + \frac{(i + i\,G^2)/2}{z + i} -  \frac{(i + i\,G^2)/2}{z - i}.
\]
\end{remark}

\section{Comparison of the three quadrature rules}\label{sec:comp}

We are now in the position to assess the three discussed rational functions
$r_m^{(G)}$, $r_m^{(T)}$, and $r_m^{(Z)}$ in view of their performance within
the FEAST method for computing eigenpairs.
A main tool will be Theorem~\ref{thm:subspace-iteration}, which allows
us to characterize the convergence of FEAST in terms of the underlying rational function.
Again assume that all eigenvalues are ordered such that for a given rational function $r_m$ we have \eqref{eq:order}. 
We also assume that a number of $\ell\leq n$ wanted eigenvalues 
$\lambda_1,\lambda_2,\ldots,\lambda_\ell$ 
of $M$ are scaled and shifted to be contained
in the interval $[-G,G]\subset (-1,1)$ for some gap parameter $G<1$. 
Finally, assume that the eigenvalues $\lambda_{n+1},\lambda_{n+2},\ldots,\lambda_N$, which are those outside the search interval not ``covered'' by the
$n$-dimensional  search space, are contained in the set $( -\infty,-G^{-1} ] \cup
[G^{-1},+\infty )$. Such a situation can always be achieved for an appropriately chosen $G$  provided that $r_m$ can separate $\lambda_\ell$ and $\lambda_{n+1}$, i.e., $|r_m(\lambda_{n+1})|>|r_m(\lambda_\ell)|$.
We can then compute for each quadrature rule and parameter $m$ the shape parameter $S>1$ (for 
the Gauss and trapezoid rules), or parameter $R>1$ (for the Zolotarev case),  so that the \emph{worst-case convergence factor}
\begin{equation}\label{eq:worst}
    \mathrm{factor}_\mathrm{worst}(m,G) = \frac{\max_{z \in (-\infty,-G^{-1}]\cup [G^{-1},+\infty)} |r_m(z)| }{ \min_{z \in [-G,G]} |r_m(z)| }
\end{equation}
is smallest possible. 

In Table~\ref{tab:comp} we show a comparison of the worst-case convergence factors for various values of $G$ and $m$. In practice,  the gap paramter $G$ is of course unknown  so that we better consider a whole range of this parameter. As can be seen for all gap parameters $G$ listed in Table~\ref{tab:comp}, the optimal worst-case convergence factors of the Zolotarev rule consistently outperform those obtained via trapezoid and Gauss quadrature  (with Gauss being slightly better than trapezoid). Let us discuss this table in some more detail. 

\paragraph{Trapezoid rule} 

For the trapezoid rule \eqref{eq:traprule}, a ``natural'' choice of the parameter~$S$ (in Table~\ref{tab:comp} denoted as ``$S=nat$'') is to achieve equioscillation on $[-G,G]$, and by Lemma~\ref{lem:trapez} 
this means that $2/(S+S^{-1}) = G$ should be satisfied. Due to  
the strictly monotone decay of $r_m^{(T)}$ outside the interval $[-G,G]$ of equioscillation, the maximum in \eqref{eq:worst} is always attained at $z = \pm G^{-1}$ and the worst-case convergence factor is given by
\[
    \mathrm{factor}_\mathrm{worst}^{(T)}(m,G) = \frac{r_m^{(T)}(G^{-1})}{r_m^{(T)}(G)} = 
\frac{\alpha + \beta}{\alpha + \beta T_{2m}\big(\frac{S + S^{-1}}{2} G^{-1}\big)}
	= \frac{\alpha+\beta}{\alpha+\beta T_{2m}(G^{-2})},
\]
with $\alpha$ and $\beta$ defined in Lemma~\ref{lem:trapez}.

However, this ``natural'' choice does \emph{not} necessarily minimize \eqref{eq:worst}, see also Table~\ref{tab:comp}. We observed numerically that \eqref{eq:worst} decreases  monotonically when $S\to 1$. However, taking $S$ very close to $1$ may be problematic from a numerical point of view because it means that the ellipse $\Gamma_S$ given by \eqref{eq:Gamma} degenerates to an interval. This means that the poles of $r_m^{(T)}$, which lie on $\Gamma_S$, come potentially close to the wanted eigenvalues, rendering the shifted linear systems in FEAST ill-conditioned or even singular. In our numerical minimization of  \eqref{eq:worst} for finding $S$ we have therefore enforced the constraint $S\geq 1.01$. In most cases reported in Table~\ref{tab:comp} the optimum for \eqref{eq:worst} was attained for $S=1.01$ (or $S$ being very close to this value). 

\paragraph{Gauss rule} Due to the irregular behaviour of $r_m^{(G)}$ defined in~\eqref{eq:gaussrule} it appears difficult to make a direct link between the gap parameter $G$ and the optimal shape parameter $S$ in the case of Gauss quadrature. For given $m$ and $G$ we have therefore computed the optimal parameter $S$ by minimizing  \eqref{eq:worst} numerically (in Table~\ref{tab:comp} denoted as ``$S=opt$''). Again, similar to the case for the trapezoid rule, we find that the optimal value for $S$ is very close to $1$, causing the  ellipse $\Gamma_S$ given by \eqref{eq:Gamma} to be very close to the search interval.

\paragraph{Zolotarev rational function} Corollary~\ref{cor:zolo} tells us that the interval of equioscillation of $r_m^{(Z)}$ about the value~1 is $[-G,G]$ when $R$ is chosen such that $G = (\sqrt{R}-1)/(\sqrt{R}+1)$ (see also \eqref{eq:mzolo}). In Table~\ref{tab:comp} this choice is denoted as ``$R=opt$''. 
Moreover, using the error bounds in that same corollary, the worst-case convergence factor can be bounded from above as
\[
    \mathrm{factor}_\mathrm{worst}^{(Z)}(m,G) = \frac{r_m^{(Z)}(G^{-1})}{r_m^{(Z)}(G)} = \frac{ E_m' }{ 1- E_m' }
	\leq \frac{2\rho^m}{1 - 2\rho^m}.
\]

\begin{table}
\caption{\label{tab:comp}Worst-case convergence factors \eqref{eq:worst} for various parameter gaps $G$ and (half-) degrees $m$.}
\begin{small}
\begin{center}
% use packages: array
\begin{tabular}{|c|c|c|c|c|c|c|c|}
\hline 
  \multirow{2}{*}{$G$}    &     \multirow{2}{*}{$m$}   &  \multicolumn{3}{|c|}{Trapezoid}  & \multicolumn{2}{|c|}{Gauss} & \multicolumn{1}{|c|}{Zolotarev} \\
  &     &  $S=\infty$ & $S = nat$ & $S = opt$ & $S=\infty$ & $S = opt$ & $R = opt$ \\ \hline \hline
      &    3  &  8.86e-1 & 6.01e-1 & 6.02e-1 (1.22)  &  8.15e-1 & 5.43e-1 (1.41)  &  1.36e-1 \\ 
      &    6  &  7.85e-1 & 3.15e-1 & 3.00e-1 (1.02)  &  4.96e-1 & 3.40e-2 (1.22)  &  7.46e-3 \\  
      &    9  &  6.95e-1 & 1.89e-1 & 1.01e-1 (1.01)  &  2.13e-1 & 5.24e-3 (1.01)  &  4.51e-4 \\ 
0.98  &   12  &  6.16e-1 & 1.18e-1 & 3.14e-2 (1.01)  &  4.83e-2 & 1.07e-3 (1.13)  &  2.74e-5 \\ 
      &   15  &  5.45e-1 & 7.38e-2 & 9.52e-3 (1.01)  &  2.37e-2 & 6.55e-5 (1.08)  &  1.67e-6 \\ 
      &   30  &  2.98e-1 & 6.24e-3 & 2.39e-5 (1.01)  &  1.06e-3 & 7.89e-10 (1.06)  &  9.73e-13 \\ 
      &   40  &  1.99e-1 & 1.16e-3 & 4.50e-7 (1.01)  &  5.38e-5 & 1.56e-13 (1.06)  &  1.23e-16 \\ \hline
      &    3  &  9.88e-1 & 9.33e-1 & 9.33e-1 (1.07)  &  9.80e-1 & 9.23e-1 (1.28)  &  3.58e-1 \\ 
      &    6  &  9.76e-1 & 7.84e-1 & 7.84e-1 (1.07)  &  9.33e-1 & 6.64e-1 (1.19)  &  4.23e-2 \\ 
      &    9  &  9.65e-1 & 6.29e-1 & 6.29e-1 (1.07)  &  8.63e-1 & 1.43e-1 (1.01)  &  5.83e-3 \\ 
0.998  &   12  &  9.53e-1 & 5.03e-1 & 5.04e-1 (1.07)  &  7.75e-1 & 1.89e-3 (1.06)  &  8.26e-4 \\ 
      &   15  &  9.42e-1 & 4.09e-1 & 4.09e-1 (1.07)  &  6.76e-1 & 1.17e-3 (1.11)  &  1.18e-4 \\ 
      &   30  &  8.87e-1 & 1.79e-1 & 8.89e-2 (1.01)  &  2.06e-1 & 5.63e-6 (1.03)  &  6.87e-9 \\ 
      &   40  &  8.52e-1 & 1.10e-1 & 2.73e-2 (1.01)  &  3.98e-2 & 6.14e-8 (1.03)  &  1.05e-11 \\ \hline
      &    3  &  9.99e-1 & 9.93e-1 & 9.93e-1 (1.02)  &  9.98e-1 & 9.92e-1 (1.26)  &  6.32e-1 \\ 
      &    6  &  9.98e-1 & 9.72e-1 & 9.72e-1 (1.02)  &  9.93e-1 & 9.51e-1 (1.15)  &  3.81e-2 \\ 
      &    9  &  9.96e-1 & 9.40e-1 & 9.40e-1 (1.02)  &  9.85e-1 & 8.60e-1 (1.10)  &  2.31e-2 \\ 
0.9998  &   12  &  9.95e-1 & 8.98e-1 & 8.99e-1 (1.02)  &  9.75e-1 & 7.36e-1 (1.09)  &  5.09e-3 \\ 
      &   15  &  9.94e-1 & 8.51e-1 & 8.52e-1 (1.02)  &  9.62e-1 & 5.94e-1 (1.08)  &  1.14e-3 \\ 
      &   30  &  9.88e-1 & 6.07e-1 & 6.10e-1 (1.02)  &  8.60e-1 & 1.66e-3 (1.04)  &  6.44e-7 \\ 
      &   40  &  9.84e-1 & 4.81e-1 & 4.83e-1 (1.02)  &  7.66e-1 & 3.71e-4 (1.02)  &  4.41e-9 \\ \hline
      &    3  &  1.00 & 9.99e-1 & 9.99e-1 (1.01)  &  1.00 & 9.99e-1 (1.26)  &  1.00 \\ 
       &    6  &  1.00 & 9.97e-1 & 9.97e-1 (1.01)  &  9.99e-1 & 9.95e-1 (1.15)  &  2.15e-1 \\ 
       &    9  &  1.00 & 9.94e-1 & 9.94e-1 (1.01)  &  9.99e-1 & 9.84e-1 (1.09)  &  5.55e-2 \\ 
0.99998  &   12  &  1.00 & 9.89e-1 & 9.89e-1 (1.01)  &  9.97e-1 & 9.65e-1 (1.07)  &  1.59e-2 \\ 
         &   15  &  9.99e-1 & 9.82e-1 & 9.83e-1 (1.01)  &  9.96e-1 & 9.35e-1 (1.06)  &  4.67e-3 \\ 
         &   30  &  9.99e-1 & 9.34e-1 & 9.38e-1 (1.01)  &  9.85e-1 & 6.60e-1 (1.04)  &  1.08e-5 \\ 
         &   40  &  9.98e-1 & 8.89e-1 & 8.99e-1 (1.01)  &  9.74e-1 & 2.21e-1 (1.01)  &  1.90e-7 \\ \hline
\end{tabular}
\end{center}
\end{small}
\end{table}

\begin{remark}
To also appreciate the fact that the Gauss and trapezoid rational approximants decay for $|z|\to\infty$, whereas Zolotarev
equioscillates, we could define another parameter $G_\mathrm{eff} \geq G^{-1}$, and compute the
\emph{effective convergence factor}
\
\[
    \mathrm{factor}_\mathrm{eff}(m,G,G_\mathrm{eff}) =  \frac{ \max_{z \in (-\infty,-G_\mathrm{eff}]\cup [G_\mathrm{eff},+\infty)} |r_m(z)| }{ \min_{z \in [-G,G]} |r_m(z)| }.
\]
The parameter $G_\mathrm{eff}$ corresponds to the modulus of the first unwanted eigenvalue outside the search interval that is not captured by the $n$-dimensional search space. As the Zolotarev rational function is equioscillating towards infinity\footnote{By dropping the absolute term in the partial fraction expansion of $r_m^{(Z)}$, which is precisely of modulus $E_m'$, this rational function could also be forced to decay for $|z|\to\infty$.}, $ \mathrm{factor}_\mathrm{eff}(m,G,G_\mathrm{eff})$ will  be equal to $\mathrm{factor}_\mathrm{worst}(m,G)$ independently of $G_\mathrm{eff}$. On the other hand, for the trapezoid and Gauss rules, $\mathrm{factor}_\mathrm{eff}(m,G,G_\mathrm{eff})$ will decrease as  $G_\mathrm{eff}$ increases.

In practice, however, it is very difficult to get a hand on $G$ and $G_\mathrm{eff}$, and even the number of wanted eigenvalues, in the first place. It is therefore problematic to rely on the faster convergence that FEAST could potentially exhibit using the trapezoid or Gauss quadrature rules with a sufficiently large $n$-dimensional search space. An exceptional case is when storage and communication are not an issue and $n$ can be made (much) larger than the number of wanted eigenvalues~$\ell$.

The Zolotarev rule, on the other hand, makes FEAST robust in the sense that the convergence factor is independent of $G_\mathrm{eff}$, and in the following we will  discuss this property in view of the load balancing problem.
\end{remark}

%\SG{Discuss: Alternatively to optimizing $S$ for Trapez and Gauss, we could just compare Zolotarev to Trapez and Gauss on the circle. In this case all rules have quadrature nodes on the circle.}

\section{Load balancing over interval partitions}\label{sec:load}

In order to achieve \emph{perfect} load balancing when using FEAST in parallel over multiple search intervals  
one would need to know in advance the number of wanted eigenvalues in each search interval, and then  
distribute the available parallel resources accordingly. 
This problem requires information about the eigenvalue distribution over 
the whole spectrum of interest, and this information is often not available (though we also mention the possibility 
to use stochastic estimates; see, e.g., \cite{DiNapoli13}).
Instead we assume here that the location of the search intervals as well as an estimate for the number of eigenvalues in each interval are given. 
Our goal is then to obtain fast convergence  within approximately the 
same number of FEAST iterations on each search interval. Some problems related to dissecting the FEAST method and choosing
the subspaces appropriately have been discussed in  \cite{galgon2011feast,2013_kraemer}.

In the original FEAST publication \cite{2009_polizzi} it was suggested to use a subspace size of $\times 1.5$ 
the estimated number of eigenvalues $\ell$ inside a given search interval (i.e., $n=1.5\ell$) to obtain a fast convergence 
 using at least $8$ nodes for the Gauss quadrature rule along a semi-circle. 
While these choices for $n$ and $m$  worked well for a large number of examples,
 it is now well understood from the discussions in the previous sections that they are also far from optimal.
In particular, we note the following two limiting cases for the choice of $n$ when the number of
contour points $m$ stays fixed:
\begin{itemize}
\item If $n$ is chosen too small (but still $n\geq\ell$), $\lambda_{n+1}$ could be located very close to the edges 
of the search interval.
This situation is likely to occur, e.g., if the eigenvalues are not evenly distributed and particularly dense 
 at the edges of the interval. By Theorem~\ref{thm:subspace-iteration} the convergence is expected to be poor in particular when $|r_m(z)|$ 
does not decay quickly enough near the search interval
and this is the case with the Gauss and trapezoid quadrature 
 rules; see Figure~\ref{fig_res2}. 
The  dependence of convergence on the value $|r_m(\lambda_{n+1})|$ 
 can cause  difficulties for achieving load balancing when FEAST runs on several search intervals in parallel. 
While on some search intervals the method may converge in, say, 2 to 3 iterations 
because $|r_m(\lambda_{n+1})|\leq 10^{-5}$, other intervals could require much longer if the subspace size $n$ 
is not sufficiently large.
\item If $n$ is chosen too large,  $|r_m(\lambda_{n+1})|$  is likely to be very small and hence convergence would be
rapid (see Figure~\ref{fig_res2} in the case of $m=8$).
 However, the location of $\lambda_{n+1}$ is not known a priori and using a  very large search space leads to three major problems: (i) 
a considerable increase in computation time since $n$ represents the number of right-hand sides to solve for each shifted linear system, (ii) an increase in communication cost as $n$ vectors need to be communicated to a master processor at each iteration, and (iii) a possibly highly rank-deficient search subspace $Z_k$  which may prevent the 
reduced pencil to be constructed stably %for such large subspace size $n$ without %re-sizing using a smaller $n$ 
%or alternatively using 
without explicit orthogonalization of $Z_k$. 
\end{itemize}

Figure~\ref{fig_res2} shows that the Zolotarev rational function  (using  $R=1e6$) has a much steeper slope
near the edges of the search interval than the Gauss and trapezoid rules (see, in particular, the magnified part). 
As a result, the a-priori known convergence factor for Zolotarev can often be obtained using a very small subspace size $n$ ($n\geq \ell$ with $n\simeq\ell$). Moreover, the convergence factor will stay almost constant if $n$ is increased further. In this sense the convergence of FEAST with the Zolotarev rule is predictable and robust. In fact, the a priori knowledge of the convergence factor can be useful for detecting whether the subspace size is chosen large enough: After a few FEAST iterations one calculates an \emph{approximate  convergence factor}  from the decrease of the eigenvector residuals. If the approximate convergence factor is much worse than expected from  Theorem~\ref{cor:zolo} (see also Table~\ref{tab:comp}) then $n$ should be increased. 

%From Table~\ref{tab:comp} one also notes that better Zolotarev %convergence factors 
%can be obtained by increasing the number of quadrature nodes $m$.

%In contrast to Gauss and Trapezoidal, however,
%and  increasing $n$  would not likely improve the convergence 
%rate that reaches an optimal value  (which can be calculated a priori). 

\begin{figure}[ht]
\begin{center}
\vspace*{5mm}
\includegraphics[width=0.60\linewidth]{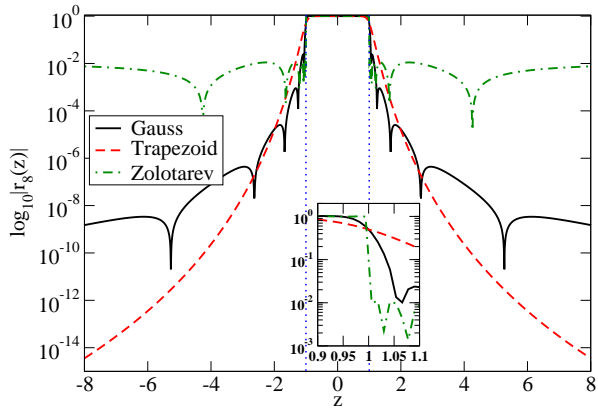}
\caption{\label{fig_res2} 
Comparison of moduli of the rational function $r_8(\lambda)$ for the Gauss ($S=\infty$), trapezoid ($S=\infty$),  
and Zolotarev ($R=1e6$) rules over  
a larger $z$-range than the one used in Figures \ref{fig:gauss}, \ref{fig:trapez}, and \ref{fig:zolotarev}.
 The magnified part focuses  on the variation of   
the rational functions near the edge $|z|=1$ of the search interval (we also recall that all these functions take the value $1/2$ at $|z|=1$). }
\end{center}
\end{figure}

In summary, the delicate choice of $n$ to achieve load balancing 
and computational efficiency is greatly simplified with the Zolotarev approach.
In practice one can achieve a uniform convergence behaviour over multiple search intervals by simply covering the region of interest by translated Zolotarev rational functions $r_m^{(Z)}(z+t)$, possibly with a small overlap. An example of three concatenated Zolotarev functions on the interval $[-3,3]$ is given in Figure~\ref{fig:load}. The expected convergence factor for all three FEAST runs will be the same (provided that $n$ is sufficiently large) and can be calculated from $G=0.98$ and $m=6$ using Theorem~\ref{cor:zolo} (in this example we  read  off from Table~\ref{tab:comp} that the expected convergence factor is $7.46e-3$).

\begin{figure}
 \centering \includegraphics[width=.75\linewidth]{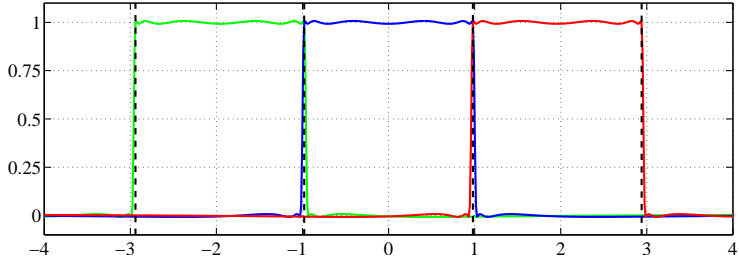}\\[-0mm]
 \caption{Translated rational functions $r_m^{(Z)}(z + 2jG)$, $j\in\{-1,0,1\}$, covering a larger interval $[-3G,3G]$. In this example  we have chosen $G=0.98$ and $m=6$.}\label{fig:load}
\end{figure}

%Ideally, a perfect load balancing providing a uniform convergence %rate over all search intervals 
%could be achieved by scaling the Zolotarev rational function %$r_m^{(Z)}$ on a slightly elongated interval. 
%In particular, using $G<1$ such as $G=0.98$ (for example), the neighbouring intervals could overlap and lead to 
%an uniform convergence rate. % (as long as $n$ is chosen sufficently large enough).
%as depicted in Figure \ref{fig:load}. 
%
%In practice, however, Zolotarev is already expected to produce a very satisfactory  convergence rate using 
%a very modest subspace size, and overlapping the search intervals to achieve load balancing would rarely 
%be needed. Additionally, the choice of an unnecessary large subspace size $n$ as initial guess for Zolotarev, 
%can be truncated as soon as one FEAST subspace iteration is performed without affecting the convergence rate.
%While Gauss or trapezoidal are still capable to produce much  better convergence rates using a
% large enough subspace size, Zolotarev is expected to provide consistency in the results which can significantly improve load balancing when 
%multiple search intervals run in parallel.

\section{Numerical experiments}\label{sec:numex}

In this section we discuss three numerical experiments stemming from  electronic structure calculations and aiming to compare the robustness and efficiency of FEAST running with Gauss and trapezoid  quadrature, as well as the Zolotarev rational function, respectively.  
All results have been obtained using the sparse solver interface of   
FEAST v2.1 \cite{FEASTsolver}, which has  been modified 
for this article to integrate the Zolotarev nodes and weights for the parameter $R=1e6$ (cf.~Section~\ref{sec:zolo}).
All numerical quadratures with the Gauss and trapezoid rules have been performed along a semi-circle (i.e., $S=\infty$; cf.~Section~\ref{sec:quad}), taking advantage of the eigenvalue counting  approach introduced in FEAST v2.1 \cite{Tang13} (this approach requires the value of $r_m$ to be $1/2$ at the interval endpoints, see also Remark~\ref{rem:symm}).   
Finally,  all spectral values $\lambda$ (including the edges of the search interval) 
are implicitly stated in the physical unit of electron Volt (for consistency with the unscaled matrix data  
all numerical values should be multiplied by the electron charge $q=1.602176\times 10^{-19}$).
 
%%%%%%%%%%%%%%%%%%%%%%%%%%%%%%%%%%%%%%%%%%%%%%%%%%%%%%%%%%%%%%%%%%%%%%%%%%%%%%%%%%%%%%%%%%%

%%%%%%%%%%%%%%%%%%%%%%%%%%%%%%%%%%%%%%%%%%%%%%%%%%%%%%%%%%%%%%%%%%%%%%%%%%%%%%%%%%%%%%%%%%%

%%%%%%%%%%%%%%%%%%%%%%%%%%%%%%%%%%%%%%%%%%%%%%%%%%%%%%%%%%%%%%%%%%%%%%%%%%%%%%%%%%%%%%%%%%%

\subsection{Example I}

Let us first consider the {\em cnt}  matrix which was presented
in \cite{2009_polizzi} and can be found in the FEAST package \cite{FEASTsolver,FEASTdoc}.
This matrix stems from a 2D FEM discretization of the
DFT/Kohn--Sham equations at a cross-section of a
(13,0) Carbon nanotube (CNT) \cite{zhang08}.
The corresponding eigenproblem takes the generalized form \eqref{eq:gep} with $A$ 
real symmetric and $B$ symmetric positive definite. The size of both matrices
is $N = 12, 450$ and their sparsity patterns are identical with
a number of $nnz = 86, 808$ nonzero entries. We are looking for the $\ell=100$ eigenvalues contained 
in the search interval $[\lambda_\mathrm{min}=-65,\lambda_\mathrm{max}=4.96]$. 

Figure~\ref{cnt_38residual} shows the residual norms (more precisely, the maximum among all residual norms of all approximate eigenpairs in the search interval) at each FEAST iteration using three ($m=3$) and eight ($m=8$) integration nodes for both Gauss, trapezoid,  
and Zolotarev. We observe poor convergence with  Gauss and 
trapezoid using a subspace of small dimension $n=102$, i.e., $n=\ell+2$. As expected the convergence with Gauss and trapezoid systematically improves 
when the subspace size $n$ is increased. Zolotarev, on the other hand, converges robustly even with $n=102$ and remains to converge at the same rate when $n$ is increased further.

\begin{figure}[htbp]
\begin{center}
\includegraphics[width=0.7\linewidth]{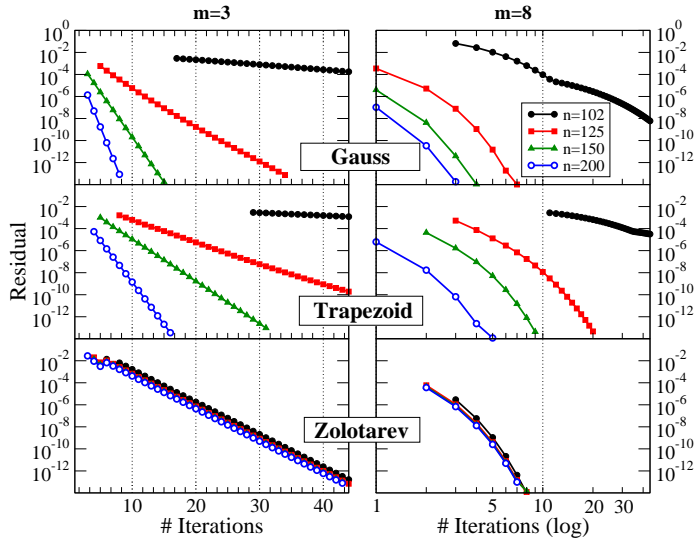}
\end{center}
\caption{\label{cnt_38residual}
FEAST residual convergence for the {\em cnt} matrix using Gauss, 
trapezoid, and Zolotarev. We have used $m=3$ (left) and $m=8$ (right) nodes while varying the subspace size $n\in\{102,125,150,200\}$. The residual norms are reported starting with the iteration where the number 
of eigenvalues in the search interval stabilizes at $100$.}
\end{figure}

A more detailed comparison between Gauss and Zolotarev for $m=8$ is provided in Figure~\ref{cnt_rational}. 
%The values of the rational functions obtained for different   subspace sizes $n$ and evaluated at $\lambda_{n+1}$ are all consistent with the corresponding convergence rate in Figure~\ref{cnt_38residual}. 
Clearly $n=102$ is insufficient for the Gauss rule to achieve a small value $|r_8^{(G)}(\lambda_{n+1})|$, 
but for larger subspace sizes this value decreases and hence Gauss converges faster. As $|r_8^{(Z)}(\lambda_{103})|$ is sufficiently small, Zolotarev attains its theoretical convergence factor of $1.12\times 10^{-2}$ (calculated using the 
results of Section \ref{sec:comp}) for $n=102$.

\begin{figure}[htbp]
\begin{center}
\includegraphics[width=0.7\linewidth]{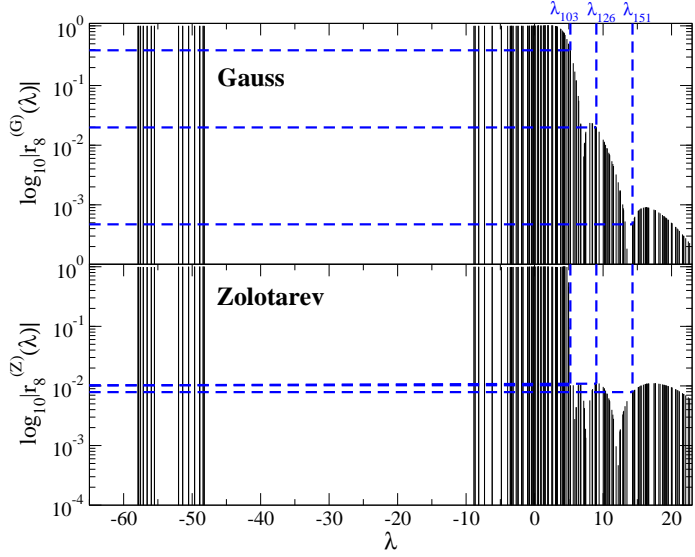}
\caption{\label{cnt_rational} Moduli of  the rational functions $r_8^{(G)}$ and $r_8^{(Z)}$ at the eigenvalues  
of the  {\em cnt} matrix in the search interval $[\lambda_\mathrm{min}=-65,\lambda_\mathrm{max}=4.96]$.
The moduli of the rational functions are given by the height of the vertical lines, and the $\lambda$-position indicates the eigenvalue. There are exactly $100$ eigenvalues 
located in the search interval. The horizontal lines provide
information about the moduli of the rational functions evaluated at $\lambda_{n+1}$ for various subspace sizes 
$n\in\{ 102,125,150\}$.}
\end{center}
\end{figure}

%%%%%%%%%%%%%%%%%%%%%%%%%%%%%%%%%%%%%%%%%%%%%%%%%%%%%%%%%%%%%%%%%%%%%%%%%%%%%%%%%%%%%%%%%%%

%%%%%%%%%%%%%%%%%%%%%%%%%%%%%%%%%%%%%%%%%%%%%%%%%%%%%%%%%%%%%%%%%%%%%%%%%%%%%%%%%%%%%%%%%%%

%%%%%%%%%%%%%%%%%%%%%%%%%%%%%%%%%%%%%%%%%%%%%%%%%%%%%%%%%%%%%%%%%%%%%%%%%%%%%%%%%%%%%%%%%%%

\subsection{Example II}

We now present a case where even a large subspace size of $n=1.5\ell$ is not enough to yield satisfactory 
FEAST convergence with Gauss quadrature. This generalized eigenproblem, {\em Caffeinep2}, is obtained
from a 3D quadratic FEM discretization of the Caffeine molecule ($\rm C_{8}H_{10}N_4O_2$),
using an all-electron DFT/Kohn--Sham/LDA model \cite{levin12,gavin13}.
The size of both matrices $A$ and $B$
is $N = 176, 622$ and their sparsity patterns are identical with
$nnz = 2,636,091$ nonzero entries.
The eigenvalues can be classified into so-called core, valence, and extended/conduction electron states. We are here searching for the first $\ell=57$ eigenvalues contained 
in the  interval $[\lambda_\mathrm{min}=-711,\lambda_\mathrm{max}=-0.19]$ covering the three physical state regions. 

Figure~\ref{Caffeinep2_rational} shows the moduli of the Gauss and Zolotarev rational functions with $m=8$. For Gauss, the choice of 
two subspace sizes $n=71$ (i.e., $n\simeq 1.25\ell$) and $n=85$ (i.e., $n\simeq 1.5\ell$) are highlighted  
with their corresponding values $|r_m^{(G)}(\lambda_{n+1})|$. 
Both values make us expect very poor convergence
factors for FEAST, and indeed after $41$ iterations the residual norms are found to have decreased only to $2.4\times 10^{-5}$ ($n=71$) and $4.8\times 10^{-6}$ ($n=85$), respectively.
For Zolotarev the figure indicate that the theoretical converge rate is already attained using a subspace size of $n=71$, in which case FEAST converges  within $9$ iterations to a residual norm of $7.8\times10^{-14}$. With $n=59$ Zolotarev-FEAST converges to about the same residual norm in $23$ iterations. 

%In contrast to the results obtained in Example I,  the selection of   a very small subspace $n=\ell+2$ (here $n=59$) for Zolotarev, is not always optimal. 

From the results in Examples I and II we  conclude that the suggested choice of $n=1.5\ell$ for Gauss (see \cite{2009_polizzi}) 
is capable of providing good convergence rates but it  lacks  robustness.
The initial choice of a subspace size  $n=1.5\ell$  with Zolotarev will  typically be safer in practice.
Note that the subspace size can easily be truncated after the first few FEAST iterations without affecting the theoretical convergence factor.

\begin{figure}[htbp]
\begin{center}
\includegraphics[width=0.7\linewidth]{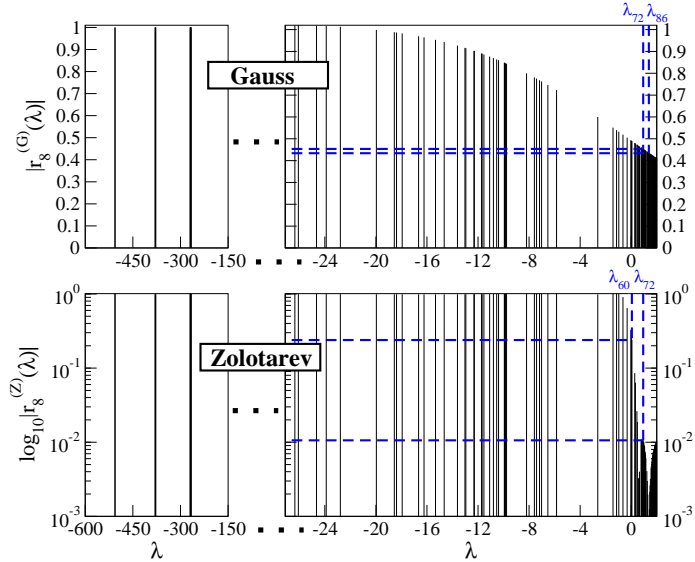}
\caption{\label{Caffeinep2_rational} Moduli of  the rational functions $r_8^{(G)}$ and $r_8^{(Z)}$ at the eigenvalues  
 of the  {\em Caffeinep2} matrix in the search interval $[\lambda_\mathrm{min}=-711,\lambda_\mathrm{max}=-0.19]$.
The moduli of the rational functions are given by the height of the vertical lines, and the $\lambda$-position indicates the eigenvalue. For visual clarity we have removed from the plots the large gap
between the core eigenvalues (part on the left) and valence/conduction eigenvalues (part on the right), where
no eigenvalues are found. The function values  for Gauss (in the top) are  plotted 
in linear scale while the ones for Zolotarev (in the bottom) are  plotted in logarithmic scale.
There are $57$ eigenvalues located in the search interval, and the horizontal lines provides
information about the moduli of the rational functions evaluated at $\lambda_{n+1}$ for various subspace sizes 
$n$, namely $n\in\{ 71, 85\}$ for Gauss and $n\in\{  59, 71\}$ for Zolotarev.}
\end{center}
\end{figure}

%%%%%%%%%%%%%%%%%%%%%%%%%%%%%%%%%%%%%%%%%%%%%%%%%%%%%%%%%%%%%%%%%%%%%%%%%%%%%%%%%%%%%%%%%%%

%%%%%%%%%%%%%%%%%%%%%%%%%%%%%%%%%%%%%%%%%%%%%%%%%%%%%%%%%%%%%%%%%%%%%%%%%%%%%%%%%%%%%%%%%%%

%%%%%%%%%%%%%%%%%%%%%%%%%%%%%%%%%%%%%%%%%%%%%%%%%%%%%%%%%%%%%%%%%%%%%%%%%%%%%%%%%%%%%%%%%%%

\subsection{Example III}

With the matrix {\em Caffeinep2} from Example II we now evaluate the efficiency of Zolotarev in terms of  
load balancing when using two search intervals. The first one, $[-711,-4]$, captures the $\ell_1=51$ core and
 valence electron states while the second one, $[-4,1.995]$, captures the first $\ell_2=55$ extended/conduction 
electron states.

Table~\ref{tab:load} reports the number of FEAST iterations needed 
to converge to a residual norm below $10^{-13}$ 
using both Gauss and Zolotarev with a subspace size of $n_j\simeq1.5\ell_j$ for the two intervals, 
i.e., $n_1=76$ and $n_2=83$. 
The results indicate that the number of FEAST iterations required on different search intervals can differ significantly using Gauss, whereas Zolotarev is capable of providing reliable load balancing. 
This is consistent with the discussions in Section~\ref{sec:load}.

%\begin{table}[htbp]
%\begin{center}
%\begin{tabular}{lcc|cc}
% \multicolumn{1}{c}{} & 
%\multicolumn{2}{c}{Gauss} & \multicolumn{2}{c}{Zolotarev}  \\
%\multicolumn{1}{c}{Intervals} & \multicolumn{1}{c}{$m=8$} & \multicolumn{1}{c}{$m=16$} & $m=8$ &  $m=16$ \\ \hline  \hline
%$[-711,-4]$ & $39$    &  $9$   & $8$ & $4$       \\ \hline
%$[-4,1.995]$  & $5$    &  $3$   & $9$ & $4$      \\ \hline\hline
%\end{tabular}
%\end{center}
%\caption{\label{tab:load} Number of FEAST iterations obtained using Gauss and Zolotarev with
%two search intervals for the {\em Caffeinep2} example. Two cases $m=8$ and $m=16$ are considered. 
%The number of eigenvalues by intervals is   $\ell_1=51$ and $\ell_2=55$, and the size of the search subspaces 
%has been set to $n_1=76$ and $n_2=83$.}
%\end{table}
\begin{table}[htbp]
\caption{\label{tab:load} Number of required FEAST iterations for the {\em Caffeinep2} example using Gauss and Zolotarev rules on 
two search intervals. Three cases $m=8$, $m=16$, and $m=32$ are considered. 
The number of eigenvalues
in the intervals is  $\ell_1=51$ and $\ell_2=55$, and the sizes of the search subspaces 
has been set to $n_1=76$ and $n_2=83$, respectively. For $m=32$, the symbol ``$~^*$'' indicates that 
 the size of the subspace $Z_k$ has been resized to a smaller dimension by FEAST v2.1~\cite{Tang13}. 
% In this case an optimal value for the iteration count (expected equal to $1$ for Zolotarev) cannot be achieved  in double precision arithmetic without an explicit orthognalization of $Z_k$.
} 
\begin{center}
\begin{tabular}{l|ccc|ccc}
 \multicolumn{1}{c}{} & 
\multicolumn{3}{|c|}{Gauss} & \multicolumn{3}{c}{Zolotarev}  \\
\multicolumn{1}{c|}{Intervals} & \multicolumn{1}{c}{$m=8$} & \multicolumn{1}{c}{$m=16$} &  \multicolumn{1}{c|}{$m=32$}& $m=8$ &  $m=16$ &  \multicolumn{1}{c}{$m=32$}\\ \hline \hline
$[-711,-4]$ & $39$    &  $9$ & $5$  & $8$ & $4$ & $3^*$      \\ \hline
$[-4,1.995]$  & $5$    &  $3$ & $3^*$  & $9$ & $4$ &$2^*$      \\ 
\hline
\end{tabular}
\end{center}
\end{table}

\section*{Summary and future work} We have studied Zolotarev rational functions as filters in the FEAST eigensolver. We have quantified the expected Zolotarev convergence factor and compared it analytically and numerically with the convergence factors  obtained via trapezoid and Gauss quadrature. The Zolotarev rational functions possess a very steep slope at the interval endpoints which often allows for a decrease of the search space dimension. Moreover, these functions do not decay towards infinity which causes FEAST to converge at a predictable, and analytically known, rate (for a sufficiently large search space dimension). We discussed the implications in view of load balancing. The new Zolotarev rules will be part of the next FEAST release, version~3.

Several questions remain open for future work. First of all, some of the poles of the Zolotarev rational functions move very close to  the real line. The same is true for the mapped Gauss rule, and even the trapezoid rule when a flat ellipse is used as the contour. It is not clear what is the effect of these poles nearby the search interval on the accuracy of the linear system solver. We have observed numerically that the weights are approximately proportional to the imaginary parts of their associated poles so that, possibly, inaccuracies in the linear system solves are damped out. The numerical experiments performed did not indicate any problems with instability.

Another question is how the Zolotarev ``quadrature rules'' generalize to moments of higher order. We have numerically observed that the Zolotarev rules integrate higher-order moments quite accurately when a polynomial weight function is introduced in \eqref{eq:moments}. Also it may be beneficial to distribute the number of equioscillation points of the Zolotarev rational function differently, for example, placing more equioscillation points outside the search interval than inside. Such a rational function can easily be constructed using, e.g., the two-interval Zolotarev approach in \cite{DGK14}.

\section*{Acknowledgement} We are grateful to Anthony Austin, Daniel Kressner, Lukas Kr\"{a}mer, Bruno Lang, Yuji Nakatsukasa, and Nick Trefethen for useful discussions.

\bibliographystyle{abbrv}
\bibliography{stefanbib}

\end{document}